\documentclass[12pt]{article}
\usepackage{graphicx}
\usepackage[a4paper, total={7in, 8in}]{geometry}
\usepackage{amssymb,amsmath,amsthm,amsfonts}
\usepackage{mathrsfs}
\usepackage{enumitem}
\usepackage{indentfirst}
\usepackage{float}
\usepackage{listings}
\usepackage{matlab-prettifier}
\usepackage{subfig}
\usepackage{diagbox}
\usepackage{bm}
\usepackage[vlined]{algorithm2e}
\newcommand{\alignedintertext}[1]{%
  \noalign{%
    \vskip\belowdisplayshortskip
    \vtop{\hsize=\linewidth#1\par
    \expandafter}%
    \expandafter\prevdepth\the\prevdepth
  }%
}
\newcommand\ovset[2]{\genfrac{}{}{0pt}{}{#1}{#2}}
\newcommand{\att}{\mathcal{A}_\mathcal{S}}

\newcommand{\lip}{\text{Lip}}
\theoremstyle{definition}
\newtheorem{definition}{Definition}[section]
\theoremstyle{plain}
\newtheorem{theorem}{Theorem}[section]
\theoremstyle{definition}
\newtheorem{remark}{Remark}[section]
\theoremstyle{definition}

\theoremstyle{definition}

\theoremstyle{plain}
\newtheorem{proposition}{Proposition}[section]
\numberwithin{equation}{section}
\usepackage{aligned-overset}

\begin{document}
\author{Bogdan Anghelina and
    Radu Miculescu
}
\date{}
\title{Covers of fractal interpolation surfaces \\ with finite families of octahedrons}
\maketitle

\textbf{Abstract.} In our previous work, \textit{On the localization of Hutchinson-Barnsley fractals}, Chaos Solitons Fractals, 173 (2023), 113-674, we presented a method for finding a finite family of closed balls whose union contains the attractor of a given iterated function system. In this paper, for the particular framework of fractal interpolation surfaces, we provide an improved version of it. This approach is more efficient, from the computational point of view, as it is based on finding the maximum of certain sets, in contrast to the previous method which uses a sorting algorithm.\\

\\

\section{Introduction}
The fractal interpolation method is due to M. Barnsley. It is appropriate for modelling rough and irregular curves and surfaces, having applications in those areas in which there is a need to construct very complicated objects like computer graphic, metallurgy, geology, seismology etc.

P. Massopust (see \cite{massopus1}) constructed fractal interpolation surfaces over
triangular regions provided that the interpolation points of the boundary of
the region are coplanar. For connected results see \cite{liang}.

J. Geronimo and D. Hardin (see \cite{geronimo}) presented algorithms for the construction
of fractal interpolation surfaces over polygonal regions with arbitrary
interpolation points. For a generalization see \cite{zhao}.

D. Hardin and P. Massopust (see \cite{hardin}) investigated $\mathbb{R}^{m}$-valued
multivariable fractal functions.

L. Dalla (see \cite{dallas}) and H. Xie and H. Sun (see \cite{xie1}) exposed a construction of a
bivariate fractal interpolation function whose domain is a rectangle $%
    \mathbb{R}^{2}$ and the interpolation points on each of its edge are
collinear. For some generalizations of these works see \cite{feng} and \cite{malysz}.

A nice method (based on the construction of fractal interpolation functions)
to obtain fractal interpolation surfaces can be found in
    \cite{bouboulis2}. For a generalization see \cite{yun}.

For extra works dealing with generalizations of fractal interpolation
surfaces see \cite{bouboulis1}, \cite{chand1}, \cite{chand2}, \cite{drakopoulos}, \cite{feng2}, \cite{liang2}, 
        \cite{metzler}, \cite{ri1}, \cite{ri2} and \cite{ruan}.

Fractal interpolation surfaces were used to approximate surfaces of rocks,
metals (\cite{xie2}), terrains \cite{yokoya}, planets \cite{cambell} and to
compress images \cite{bouboulis3}.

In \cite{locHBfrct} we presented a method for finding a finite family of closed balls whose union contains the attractor of a given iterated functions system, and we explained its relevance. More precisely, if \(S=((X,d), (f_i)_{i\in I})\) is such a system consisting of at least two functions and \(\rho_i\in [0,\infty)\), \(i\in I\), represent a solution of the system of equations
\[\rho_i = \lip(f_i)(\max_{i,j\in I}d(e_i,e_j)+\max_{j\in I\setminus\{i\}}\rho_j)\tag{\(\star\)},\] where \(e_i\) is the unique fixed point of \(f_i\), then \[\att \subseteq \bigcup _{i\in I}B[e_i,\rho_i],\]
where \(\att\) designates the attractor of \(\mathcal{S}\).

Let us recall that in order to obtain a solution of \((\star)\) we relabeled the functions \(f_i\) such that \(\lip(f_{i_1})\leq \lip(f_{i_2}) \leq \dots \leq \lip(f_{i_n})\), i.e. we ordered the set \(\{\lip(f_i)|i\in I\}\), where \(I=\{i_1,i_2,\dots , i_n\}\). From the computational point of view this approach is not very efficient.

In this paper we provide an improved version of the above-mentioned method for the particular framework of fractal interpolation surfaces. One inconvenience that occurs in this case is the difficulty to find a concrete form of \(\lip(f_i)\) in terms of the given data set. As, in view of the work presented in \cite{dallas}, some constants \(c_i\in [0,1)\), having the property that \(d(f_i(x),f_i(y))\leq c_i d(x,y)\) for all \(x,y\in X\), are available, we remark that our previously mentioned result from \cite{locHBfrct} is still valid if one replaces \(\lip(f_i)\) with \(c_i\).

Moreover, we find a solution of the system \((\star)\) via a more efficient approach, which is based on determining \(\displaystyle\max_{i\in I}c_i\) and \(\displaystyle\max_{{j\in I}, \;{c_j\neq\max_{i\in I}c_i}} c_j\) rather than on the ordering of \(\{\lip(f_i)|i\in I\}\).

Some graphical representations for such covers of fractal interpolation surfaces are given.
\section{Preliminaries}

Given a metric space \((X,d)\), \(x\in X\) and \(r>0\) we shall use the following notation:

\begin{itemize}
    \item \(B[x,r]=\{y\in X \;|\; d(y,x)\leq r\}\)

    \item \(P_{cp}(X)=\{A\subseteq X| A \text{ is non-empty and compact}\}\)
    \item \(h\) is the Hausdorff-Pompeiu metric.

\end{itemize}

For a Lipschitz function \(f:X \to X\) we shall denote by \(\lip(f)\) the Lipschitz constant of \(f\).
\\

\begin{definition}[]\label{fractaldef} An iterated function system (for short IFS) is a pair \(\mathcal{S}=\textbf{(}(X,d),(f_i)_{i\in I}\textbf{)}\), where \((X,d)\) is a complete metric space and \(f_i:X\to X\), \(i\in I\), are Banach contractions.

    The function \(F_\mathcal{S}:P_{cp}(X)\to P_{cp}(X)\), given by \[F_\mathcal{S}(K)=\bigcup_{i\in I}f_i(K),\] for every \(K\in P_{cp}(X)\), is called the fractal operator (or the Hutchinson operator) associated with \(\mathcal{S}\).

\end{definition}
\begin{proposition}\label{2:picard} If \(\mathcal{S}=\boldsymbol{(}(X,d),(f_i)_{i\in I}\boldsymbol{)}\) is an iterated function system, then \(F_\mathcal{S}\) is a contraction with respect to \(h\). Its unique fixed point is denoted by \(\att\) and it is called the attractor of \(\mathcal{S}\) since
    \[\lim_{n\to\infty} \underbrace{(F_\mathcal{S}\circ F_\mathcal{S}\circ\dots\circ F_\mathcal{S})}_{n \text{ times }}(K) = \att,\]
    for each \(K\in P_{cp}(X).\)
\end{proposition}

\section{Covers with a finite family of closed balls of the attractor of an IFS}
We shall consider
\begin{itemize}
    \item \(\mathcal{S}=((X,d),(f_i)_{i\in \{1,\dots, n\}})\) an iterated function system consisting of contractions, with \(n\geq 2\)
    \item  \(\gamma_i\) the unique fixed point of \(f_i\), where \(i\in \{1,\dots, n\}\)
    \item \(c_i\in[0,1)\) such that \(d(f_i(x),f_i(x))\leq c_i d(x,y)\) for all \(x,y \in X\), where \(i\in \{1,\dots, n\}\)
    \item \(i'\in\{1,\dots, n\}\) such that \(c_{i'} = \max \{c_1, \dots, c_n\}\)
    \item \(i''\in\{1,\dots,n\}\setminus\{i'\}\) such that \(c_{i''}= \max\{c_i\;|\;i\in\{1,\dots,n\}\setminus\{i'\}\}\)
          \item\(M= \max_{i,j\in \{1,\dots, n\}} d(\gamma_i,\gamma_j)\).
    \item The system \(\mathfrak{S}\), with the unknowns \(\rho_1,\dots \rho_n\), consisting on the following \(n\) equations \[\rho_i=c_i\left(M+\max_{j\in \{1,\dots, n\}\setminus\{i\}}\rho_j\right),\] where \(i\in\{1,\dots,n\}\).
\end{itemize}

The arguments used on the proof of Proposition 3.3 from \cite{locHBfrct} ensure the validity of the following result:
\begin{proposition}\label{prop1}
    In the above-mentioned framework, \[\att \subseteq \bigcup_{i\in\{1,\dots, n\}} B[\gamma_i,\rho_i],\] for each solution \(\rho_1,\dots ,\rho_n\) of the system \(\mathfrak{S}\).
\end{proposition}

\begin{proposition}\label{prop2}
    In the above-mentioned framework,
    \[\rho_{i'}=Mc_{i'}\frac{1+c_{i''}}{1-c_{i'}c_{i''}}\]
    and
    \[\rho_{i}=Mc_{i}\frac{1+c_{i'}}{1-c_{i'}c_{i''}},\]
    for all \(i\in\{1,\dots,n\}\setminus\{i'\}\), represent a solution of the system \(\mathfrak{S}\).
\end{proposition}

\begin{proof}
    \item

    \textbf{Claim.} \[\rho_{i'}\geq \rho_{i''}\geq \rho_i, \] for all \(i\in\{1,\dots,n\}\setminus\{i'\}\).

    Justification of the Claim.
    We have
    \[c_{i''}(1+c_{i'})=c_{i''}+c_{i'}c_{i''}\leq c_{i'}+c_{i'}c_{i''}= c_{i'}(1+c_{i''}),\]
    so
    \[\rho_{i''}=Mc_{i''}\frac{1+c_{i'}}{1-c_{i'}c_{i''}} \leq Mc_{i'}\frac{1+c_{i''}}{1-c_{i'}c_{i''}} =\rho_{i'}.\]

    Since \(c_{i''}\geq c_{i}\) for all \(i\in\{1,\dots ,n\}\setminus\{i'\}\), we have
    \[\rho_{i}=Mc_{i}\frac{1+c_{i'}}{1-c_{i'}c_{i''}} \leq Mc_{i''}\frac{1+c_{i'}}{1-c_{i'}c_{i''}}=\rho_{i''}  ,\]
    for all \(i\in \{1,\dots,n\} \setminus \{i'\}\).
    Hence, the justification of the Claim is complete.

    Therefore, taking into account the previous Claim, we get
    
    \begin{equation}\tag{1}\max_{j\in\{1,\dots, n\}\setminus\{i'\}} \rho_j = \rho_{i''}.\end{equation}
    and
    \begin{equation}\tag{2}\max_{j\in\{1,\dots, n\}\setminus\{i\}} \rho_j = \rho_{i'},\end{equation} for all \(i\in\{1,\dots,n\}\setminus\{i'\} \).

    We have
    \begin{align*}
        \rho_{i'} & = Mc_{i'}\frac{1+c_{i''}}{1-c_{i'}c_{i''}}=Mc_{i'}\frac{1-c_{i'}c_{i''}+c_{i'}c_{i''}+c_{i''}}{1-c_{i'}c_{i''}}=c_{i'}\left(M+Mc_{i''}\frac{1+c_{i'}}{1-c_{i'}c_{i''}}\right) \\&=c_{i'}(M+\rho_{i''})\overset{(1)}{=}c_{i'}(M+\max_{j\in\{1,\dots, n\}\setminus\{i'\}} \rho_j)
    \end{align*}
    and
    \begin{align*}\rho_i &= Mc_{i}\frac{1+c_{i'}}{1-c_{i'}c_{i''}}=Mc_{i}\frac{1-c_{i'}c_{i''}+c_{i'}c_{i''}+c_{i'}}{1-c_{i'}c_{i''}}=c_{i}\left(M+Mc_{i'}\frac{1+c_{i''}}{1-c_{i'}c_{i''}}\right)\\&=c_i(M+\rho_{i'})\overset{(2)}{=}c_i(M+\max_{j\in\{1,\dots, n\}\setminus\{i\}} \rho_j),\end{align*}
    for all \(i\in\{1,\dots,n\}\setminus\{i'\}\).

    Hence, the above-mentioned values of \(\rho_i\) give a solution of the system \(\mathfrak{S}\).
\end{proof}
Combining Propositions \ref{prop1} and \ref{prop2} we get:
\begin{theorem}
    In the above-mentioned framework, we have
    \[\att\subseteq \left(\bigcup_{i\in\{1,\dots,n\}\setminus\{i'\}}B\left[\gamma_{i},Mc_{i}\frac{1+c_{i'}}{1-c_{i'}c_{i''}}\right]\right)\bigcup B\left[\gamma_{i'},Mc_{i'}\frac{1+c_{i''}}{1-c_{i'}c_{i''}}\right].\]
\end{theorem}

\begin{remark}For \(p\in\mathbb{N}\), let us consider the iterated function system
    \[\mathcal{S}_p =((X,d),(f_{i_1}\circ\dots\circ f_{i_p})_{i_1,\dots,i_p\in\{1,\dots,n\}}).\]

    Note that :
    \begin{enumerate}
        \item[(\(\alpha\))] (see Proposition 3.4 from \cite{locHBfrct}) \[\mathcal{A}_{\mathcal{S}_p}=\att.\]
        \item[(\(\beta\))] \[d((f_{i_1}\circ\dots\circ f_{i_p})(x),(f_{i_1}\circ\dots\circ f_{i_p})(y))\leq c_{i_1}\cdot\ldots\cdot c_{i_p}d(x,y),\]
            for all \(x,y\in X\) and \(i_1,\dots,i_p\in\{1,\dots,n\}\).
    \end{enumerate}

    Since \[c_{i_1}\cdot \ldots \cdot c_{i_p}\leq c_{i'}^p,\]
    and \[\lim_{p\to\infty} c_{i'}^p=0,\]
    via \(\alpha)\) and Theorem 3.1, we infer that, by increasing \(p\), \(\att\) can be covered with a finite family of closed balls having the radii as small as we want.
\end{remark}
\section{Fractal interpolation surfaces}

In this section, following \cite{dallas}, we present the basic facts concerning fractal interpolation surfaces.

Let \(I=[a,b], J=[c,d]\) and the data set
\(\mathcal{D}=\{(x_k,y_l,z_{k,l})\in\mathbb{R}^3|k\in\{0,1,\dots,n\},l\in\{0,1,\dots,m\}\}\) be such that

\[a= x_0< x_1<\dots<x_{n-1}<x_n=b,\]
\[c=y_0<y_1<\dots<y_{m-1}<y_m=d,\]
and each of the sets

\begin{align*}\tag{\(*\)}\label{collin}
     & \{(x_0, y_l, z_{0, l})|l\in\{0, 1, \dots, m\}\},  \\
     & \{(x_n, y_l, z_{n, l})|l\in\{0, 1, \dots, m\}\},  \\
     & \{(x_k, y_0, z_{k, 0})| k\in \{0, 1, \dots, n\}\}
    \intertext{and}
     & \{(x_k, y_m, z_{k, m})| k\in \{0,1,\dots,n\}\}\end{align*} consists of collinear points.

For \(k\in\{1,\dots,n\}\) and \(l\in\{1,\dots,m\}\) we define the function \(F_{k,l}:I\times J\times\mathbb{R}\to I\times J\times\mathbb{R}\), by

\[F_{k,l}(x,y,z) = (a_kx+b_k,c_ly+d_l,e_{k,l}x+f_{k,l}y+g_{k,l}z+\alpha_{k,l}xy+\beta_{k,l}),\] for all \((x,y,z)\in I\times J \times \mathbb{R}\),
where the coefficients \(a_k,b_k,c_l,d_l,f_{k,l},\alpha_{k,l}\) and \(\beta_{k,l}\) are chosen such that:

\[F_{k,l}(x_0,y_0,z_{0,0}) = (x_{k-1},y_{l-1},z_{k-1,l-1}),\]
\[F_{k,l}(x_n,y_0,z_{n,0}) = (x_{k},y_{l-1},z_{k,l-1}),\]
\[F_{k,l}(x_0,y_m,z_{0,m}) = (x_{k-1},y_{l},z_{k-1,l})\]
and
\[F_{k,l}(x_n,y_m,z_{n,m}) = (x_{k},y_{l},z_{k,l}).\]

Equivalently, we obtain the following systems of equations

\[\begin{cases}
        a_kx_0+b_k                                                          & =x_{k-1}     \\
        c_ly_0+d_l                                                          & =y_{l-1}     \\
        e_{k,l}x_0+f_{k,l}y_0+g_{k,l}z_{0,0}+\alpha_{k,l}x_0y_0+\beta_{k,l} & =z_{k-1,l-1}
    \end{cases},\]
\[\begin{cases}
        a_kx_n+b_k                                                          & =x_{k}     \\
        c_ly_0+d_l                                                          & =y_{l-1}   \\
        e_{k,l}x_n+f_{k,l}y_0+g_{k,l}z_{n,0}+\alpha_{k,l}x_ny_0+\beta_{k,l} & =z_{k,l-1}
    \end{cases},\]
\[\begin{cases}
        a_kx_0+b_k                                                         & =x_{k-1}   \\
        c_ly_m+d_l                                                         & =y_{l}     \\
        e_{k,l}x_0+f_{k,l}y_m+g_{k,l}z_{0,m}+\alpha_{kl}x_0y_m+\beta_{k,l} & =z_{k-1,l}
    \end{cases}\]
and
\[\begin{cases}
        a_kx_n+b_k                                                          & =x_{k}   \\
        c_ly_m+d_l                                                          & =y_{l}   \\
        e_{k,l}x_n+f_{k,l}y_m+g_{k,l}z_{n,m}+\alpha_{k,l}x_ny_m+\beta_{k,l} & =z_{k,l}
    \end{cases}.\]

Using the first two equations of each system we obtain
\[a_k=\frac{x_k-x_{k-1}}{x_n-x_0},\;
    b_k=\frac{x_{k-1}x_n-x_{k}x_0}{x_n-x_0},\;
    c_l=\frac{y_l-y_{l-1}}{y_m-y_0} \text{ and }
    d_l=\frac{y_{l-1}y_m-y_{l}y_0}{y_m-y_0}.\]

From the remaining equations we get
\[\begin{cases}
        e_{k,l}x_n+f_{k,l}y_m+g_{k,l}z_{n,m}+\alpha_{k,l}x_ny_m+\beta_{k,l} & =z_{k,l}     \\
        e_{k,l}x_0+f_{k,l}y_m+g_{k,l}z_{0,m}+\alpha_{k,l}x_0y_m+\beta_{k,l} & =z_{k-1,l}   \\
        e_{k,l}x_n+f_{k,l}y_0+g_{k,l}z_{n,0}+\alpha_{k,l}x_ny_0+\beta_{k,l} & =z_{k,l-1}   \\
        e_{k,l}x_0+f_{k,l}y_0+g_{k,l}z_{0,0}+\alpha_{k,l}x_0y_0+\beta_{k,l} & =z_{k-1,l-1}
    \end{cases},\] or equivalently, with the notation
\[p_{k,l}=z_{k,l}-g_{k,l}z_{n,m},\;
    q_{k,l}=z_{k-1,l}-g_{k,l}z_{0,m},\;\]
\[r_{k,l}=z_{k,l-1}-g_{k,l}z_{n,0}\text{ and }
    t_{k,l}=z_{k-1,l-1}-g_{k,l}z_{0,0},
\] we get
\[\begin{cases}
        e_{k,l}x_n+f_{k,l}y_m+\alpha_{k,l}x_ny_m+\beta_{k,l} & =p_{k,l} \\
        e_{k,l}x_0+f_{k,l}y_m+\alpha_{k,l}x_0y_m+\beta_{k,l} & =q_{k,l} \\
        e_{k,l}x_n+f_{k,l}y_0+\alpha_{k,l}x_ny_0+\beta_{k,l} & =r_{k,l} \\
        e_{k,l}x_0+f_{k,l}y_0+\alpha_{k,l}x_0y_0+\beta_{k,l} & =t_{k,l}
    \end{cases}.\]

Hence,
\[\begin{cases}
        e_{k,l}+\alpha_{k,l}y_m & =\frac{p_{k,l}-q_{k,l}}{x_n-x_0} \\
        e_{k,l}+\alpha_{k,l}y_0 & =\frac{r_{k,l}-t_{k,l}}{x_n-x_0} \\
    \end{cases}\text{ and } \begin{cases}
        f_{k,l}+\alpha_{k,l}x_n & =\frac{p_{k,l}-r_{k,l}}{y_m-y_0} \\
        f_{k,l}+\alpha_{k,l}x_0 & =\frac{q_{k,l}-t_{k,l}}{y_m-y_0} \\
    \end{cases}.\]

Therefore, for \(g_{k,l}\) arbitrarily chosen in \((0,1)\), we have
\[\begin{aligned}
        \alpha_{k,l} & =\frac{p_{k,l}-q_{k,l}-r_{k,l}+t_{k,l}}{(x_n-x_0)(y_m-y_0)}
    \end{aligned},\]
\[\begin{aligned}
        e_{k,l} & =\frac{y_0(q_{k,l}-p_{k,l})-y_m(t_{k,l}-r_{k,l})}{(x_n-x_0)(y_m-y_0)}
    \end{aligned},\]
\[\begin{aligned}
        f_{k,l} & =\frac{x_0(r_{k,l}-p_{k,l})-x_n(t_{k,l}-q_{k,l})}{(x_n-x_0)(y_m-y_0)}
    \end{aligned}\]
and
\[\begin{aligned}
        \beta_{k,l} & =\frac{y_0(x_0p_{k,l}-x_nq_{k,l})-y_m(x_0r_{k,l}-x_nt_{k,l})}{(x_n-x_0)(y_m-y_0)}
    \end{aligned}.\]

Let us consider
\[\theta_1 = \left\{\begin{aligned}
    & \qquad \qquad\qquad1,                                                                       &  & e_{1,1}=\dots = e_{n,m}=\alpha_{1,1}=\dots=\alpha_{n,m}=0; \\
    & \frac{\displaystyle 1-\max_{k\in\{1,\dots,n\}} a_k}{\displaystyle 2\max_{\ovset{k\in\{1,\dots,n\}}{l\in\{1,\dots,m\}}}(|e_{k,l}|+\delta|\alpha_{k,l}|)}, &  & \text{otherwise},
 \end{aligned}\right.\]
 \[\theta_2 = \left\{\begin{aligned}
    & \qquad \qquad\qquad1,                                                                       &  & f_{1,1}=\dots = f_{n,m}=\alpha_{1,1}=\dots=\alpha_{n,m}=0; \\
    & \frac{\displaystyle 1-\max_{l\in\{1,\dots,m\}} c_l}
    {\displaystyle 2\max_{\ovset{k\in\{1,\dots,n\}}{l\in\{1,\dots,m\}}}(|f_{k,l}|+\delta|\alpha_{k,l}|)}, &  & \text{otherwise},
 \end{aligned}\right.\]
and
\[\theta =
    \min\left\{\theta_1,\theta_2
    \right\},\]
where \[\delta:=\max\{|a|,|b|,|c|,|d|\}.\]

The function \(\rho:\mathbb{R}^3\times\mathbb{R}^3\to[0,\infty)\), given by
\[\rho((x,y,z),(x',y',z'))=|x-x'|+|y-y'|+\theta|z-z'|,\] for all \((x,y,z),(x',y',z')\in\mathbb{R}^3 \), is a metric.
\\\\

\textbf{\(F_{k,l}\) is a contraction in respect of \(\rho\)}
\\

We have
\begin{align*}
     & \rho(F_{k,l}(x,y,z),F_{k,l}(x',y',z'))                                                            \\
     & \quad=a_k|x-x'|+c_l|y-y'|+\theta|e_{k,l}(x-x')+f_{k,l}(y-y')+g_{k,l}(z-z')+\alpha_{k,l}(xy-x'y')|
    \\&\quad\leq (a_k+\theta|e_{k,l}|)|x-x'|+(c_l+\theta|f_{k,l}|)|y-y'| +\theta g_{k,l}|z-z'|+\theta|\alpha_{k,l}||x(y-y')+y'(x-x')|
    \\&\quad\leq [a_k+\theta(|e_{k,l}|+\delta|\alpha_{k,l}|)]|x-x'|+[c_l+\theta(|f_{k,l}| +\delta|\alpha_{k,l}|)]|y-y'|+\theta g_{k,l}|z-z'|
    \\&\quad\leq C_{k,l} \rho((x,y,z),(x',y',z')),
\end{align*}
for all \((x,y,z),(x',y',z')\in I \times J \times \mathbb{R}\), \(k\in\{1,\dots,n\}\) and \(l\in\{1,\dots,m\}\), where
\[C_{k,l}=\max\left\{a_k+\theta(|e_{k,l}|+\delta|\alpha_{k,l}|),\;c_l+\theta(|f_{k,l}|+\delta|\alpha_{k,l}|),\; g_{k,l} \right\}<1.\]

Therefore, \(F_{k,l}\) is a contraction in respect of the \(\rho\) metric.\\

\textbf{Finding an IFS whose attractor is a surface which interpolates \(\mathcal{D}\)}
\\

According to Proposition 2.2 from \cite{dallas} there exists a continuous function \(f:I\times J\to \mathbb{R}\) having the following two properties:
\begin{enumerate}
    \item \[f(x_k,y_l)=z_{k,l},\]
          for all \(k\in\{0,1,\dots,n\}\) and \(l\in\{0,1,\dots,m\}\), i.e. \(f\) interpolates \(\mathcal{D}\).
    \item \[G_f=\mathcal{A}_\mathscr{S},\]
          where \(\mathscr{S}\) is the IFS \(((I\times J\times\mathbb{R},\rho),(F_{k,l})_{k\in\{1,\dots,n\},l\in\{1,\dots,m\}})\).
\end{enumerate}

In view of Subsection 2.2 from \cite{dallas}, the collinearity condition \eqref{collin} is not restrictive.
\\

\textbf{Deriving a formula for the fixed point of \(F_{k,l}\)}
\\

For \(k\{\in{1,\dots,n}\} \) and \(l\in\{1,\dots,m\}\), let \(\gamma_{k,l}=(x_{k,l},y_{k,l},z_{k,l})\) be the fixed point of \(F_{k,l}\).

Then
\[(a_kx_{k,l}+b_k,c_ly_{k,l}+d_l,e_{k,l}x_{k,l}+f_{k,l}y_{k,l}+g_{k,l}z_{k,l}+\alpha_{k,l}x_{k,l}y_{k,l}+\beta_{k,l})=(x_{k,l},y_{k,l},z_{k,l}),\]
hence
\[\begin{cases}
        a_kx_{k,l}+b_k                                                          & =x_{k,l}  \\
        c_ly_{k,l}+d_l                                                          & =y_{k,l}  \\
        e_{k,l}x_{k,l}+f_{k,l}y_{k,l}+g_{k,l}z_{k,l}+\alpha_{k,l}x_{k,l}y_{k,l}+\beta_{k,l} & = z_{k,l}
    \end{cases}.\]

From the first two equation we deduce that \[x_{k,l} = \frac{b_k}{1-a_k} \text{ and } y_{k,l} = \frac{d_l}{1-c_l}.\]

Substituting in the third equation we obtain
\[e_{k,l}\frac{b_k}{1-a_k}+f_{k,l}\frac{d_l}{1-c_l}+g_{k,l}z_{k,l}+\alpha_{k,l}\frac{b_kd_l}{(1-a_k)(1-c_l)}+\beta_{k,l}= z_{k,l},\]
therefore
\[z_{k,l} =\frac{1}{1-g_{k,l}}\left[e_{k,l}\frac{b_k}{1-a_k}+f_{k,l}\frac{d_l}{1-c_l}+\alpha_{k,l}\frac{b_kd_l}{(1-a_k)(1-c_l)}+\beta_{k,l}\right].\]

We conclude that \[\gamma_{k,l} = \left(\frac{b_k}{1-a_k},\frac{d_l}{1-c_l},\frac{1}{1-g_{k,l}}\left[e_{k,l}\frac{b_k}{1-a_k}+f_{k,l}\frac{d_l}{1-c_l}+\alpha_{k,l}\frac{b_kd_l}{(1-a_k)(1-c_l)}+\beta_{k,l}\right]\right).\]
\\

\textbf{Describing balls in respect of \(\rho\) }
\\

The closed ball with in respect of \(\rho\), having radius \(r>0\) and center \((x_0,y_0,z_0)\in\mathbb{R}^3\), is the set
\[\begin{aligned}\{(x,y,z)\in\mathbb{R}^3|\rho((x,y,z),(x_0,y_0,z_0))\leq r\}=\{(x,y,z)\in\mathbb{R}^3||x-x_0|+|y-y_0|+\theta|z-z_0|\leq r\}, \\
    \end{aligned}\]
denoted by \(O[(x_0,y_0,z_0),r]\). 

 This set is an octahedron having the following vertices 
\(V_1\left(x_0+r,y_0,z_0\right)\),
\(V_2\left(x_0,y_0+r,z_0\right)\),
\(V_3\left(x_0,y_0,z_0+\frac{r}{\theta}\right)\),
\(V_4\left(x_0-r,y_0,z_0\right)\),
\(V_5\left(x_0,y_0-r,z_0\right)\) and
\(V_6\left(x_0,y_0,z_0-\frac{r}{\theta}\right)\).
\\

\textbf{Obtaining a cover for \(G_f\)}
\\

Summarizing, we have
{\allowdisplaybreaks
\begin{align*}
     & \bullet \delta:=\max\{|a|,|b|,|c|,|d|\}.
     \\&\bullet\theta =
        \min\left\{\theta_1,\theta_2
        \right\}, \text{ where }
    \\& \theta_1 = \left\{\begin{aligned}
        & \qquad \qquad\qquad1,                                                                       &  & e_{1,1}=\dots = e_{n,m}=\alpha_{1,1}=\dots=\alpha_{n,m}=0; \\
        & \frac{\displaystyle 1-\max_{k\in\{1,\dots,n\}} a_k}{\displaystyle 2\max_{\ovset{k\in\{1,\dots,n\}}{l\in\{1,\dots,m\}}}(|e_{k,l}|+\delta|\alpha_{k,l}|)}, &  & \text{otherwise},
     \end{aligned}\right.
     \\& \theta_2 = \left\{\begin{aligned}
        & \qquad \qquad\qquad1,                                                                       &  & f_{1,1}=\dots = f_{n,m}=\alpha_{1,1}=\dots=\alpha_{n,m}=0; \\
        & \frac{\displaystyle 1-\max_{l\in\{1,\dots,m\}} c_l}
        {\displaystyle 2\max_{\ovset{k\in\{1,\dots,n\}}{l\in\{1,\dots,m\}}}(|f_{k,l}|+\delta|\alpha_{k,l}|)}, &  & \text{otherwise},
     \end{aligned}\right.
    \\&\bullet M=\max_{\ovset{i,k\in\{1,\dots,n\}}{j,l\in\{1,\dots,m\}}}\rho(\gamma_{i,j},\gamma_{k,l})
    \intertext{ and }
     & \bullet\gamma_{k,l} = \left(\frac{b_k}{1-a_k},\frac{d_l}{1-c_l},\frac{1}{1-g_{k,l}}\left[e_{k,l}\frac{b_k}{1-a_k}+f_{k,l}\frac{d_l}{1-c_l}+\alpha_{k,l}\frac{b_kd_l}{(1-a_k)(1-c_l)}+\beta_{k,l}\right]\right)
    \\&
    \\&\bullet C_{k,l}=\max\left\{a_k+\theta(|e_{k,l}|+\delta|\alpha_{k,l}|),\;c_l+\theta(|f_{k,l}|+\delta|\alpha_{k,l}|), g_{k,l} \right\}
\end{align*} for all \(k\in\{1,\dots ,n\} \text{ and }  l\in\{1,\dots ,m\}\).}

Let us choose \(k',k''\in\{1,\dots,n\}\) and \(l',l''\in\{1,\dots,m\}\) such that 
\[C_{k',l'}=\max_{\ovset{k\in\{1,\dots,n\}}{l\in\{1,\dots,m\}}}C_{k,l}\qquad\text{and}\qquad C_{k'',l''}=\max_{\ovset{k\in\{1,\dots,n\}\setminus\{k'\}}{l\in\{1,\dots,m\}\setminus\{l'\}}}C_{k,l}.\]

Taking into account the previous results, we get our main result:
\begin{theorem}\label{incl} In the previous framework, we have
    \[G_f\subseteq \left(\bigcup_{\ovset{k\in\{1,\dots,n\}\setminus\{k'\}}{l\in\{1,\dots,m\}\setminus\{l'\}}}O\left[\gamma_{k,l},MC_{k,l}\frac{1+C_{{k',l'}}}{1-C_{k',l'}C_{k'',l''}}\right]\right)\bigcup O\left[\gamma_{k',l'},MC_{k',l'}\frac{1+C_{{k'',l''}}}{1-C_{k',l'}C_{k'',l''}}\right].\]
\end{theorem}
The set \[\left(\bigcup_{\ovset{k\in\{1,\dots,n\}\setminus\{k'\}}{l\in\{1,\dots,m\}\setminus\{l'\}}}O\left[\gamma_{k,l},MC_{k,l}\frac{1+C_{{k',l'}}}{1-C_{k',l'}C_{k'',l''}}\right]\right)\bigcup O\left[\gamma_{k',l'},MC_{k',l'}\frac{1+C_{{k'',l''}}}{1-C_{k',l'}C_{k'',l''}}\right],\] which is a cover of \(\mathcal{A}_\mathscr{S}=G_f\) will be denoted by \(\mathcal{C_\mathscr{S}}\).

For \(p\in\mathbb{N}\), by \(\mathcal{C}_p\) we designate \(\mathcal{C}_{\mathscr{S}_p}\), which is also a cover of \(G_f\).

\begin{remark}
    In view of Remark 3.1, by increasing \(p\), the diameters of the octahedrons occurring in \(\mathcal{C}_p\) could be as small as we want.
\end{remark}

\section{Some graphical representations}

We apply our results to the data sets presented in the following examples:
\\

\textbf{Example 1}
\\

Let
\[x_0=0,\;x_1=100,\;x_2=200,\]
\[y_0=0,\;y_1=100,\;y_2=200,\]
with the corresponding \(z_{k,l}\) values 
  \begin{center}
        \begin{tabular}{ c|c|c|c }
            \backslashbox{\(l\)}{\(k\)}& 0 & 1& 2 \\ 
            \hline 0 & 0 & -10 & -20\\
            \hline 1& 10 & -30 & -10\\
            \hline 2 & 20 & 10 & 0
           \end{tabular}
    \end{center}
and \(g_{k,l}\) given by

\begin{center}
    \begin{tabular}{ c|c|c }
        \backslashbox{\(l\)}{\(k\)}& 1& 2 \\ 
        \hline 1 & 0.7 & 0.5 \\
        \hline 2& 0.6 & 0.6
    \end{tabular}.
\end{center}

The interpolation surface obtained has the following graphical representation:

\captionsetup[subfigure]{labelformat=empty}
\begin{figure}[H]
  \centering
  \begin{tabular}{c}
    \subfloat[\((O):\) Visualization of the interpolated surface]{\includegraphics[width = 0.5\textwidth]{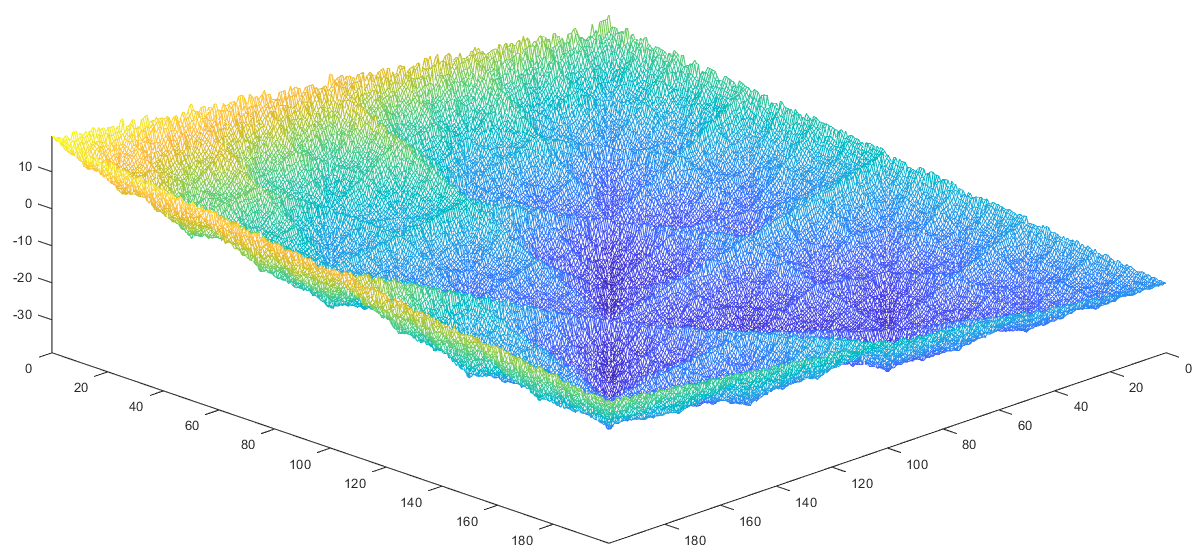}} 
  \end{tabular}
\end{figure}

The figures \((I),\) \((II),\) \((III),\) \((IV),\) \((V),\) \((VI)\) contain graphical representations of the covers \(\mathcal{C}_1\), \(\mathcal{C}_3\), \(\mathcal{C}_5\), \(\mathcal{C}_7\) and \(\mathcal{C}_9\).
\captionsetup[subfigure]{labelformat=empty}
\begin{figure}[H]
  \centering
  \begin{tabular}{cc}
    \subfloat[\((I):\) Visualization of \(\mathcal{C}_1\)]{\includegraphics[width = 0.5\textwidth]{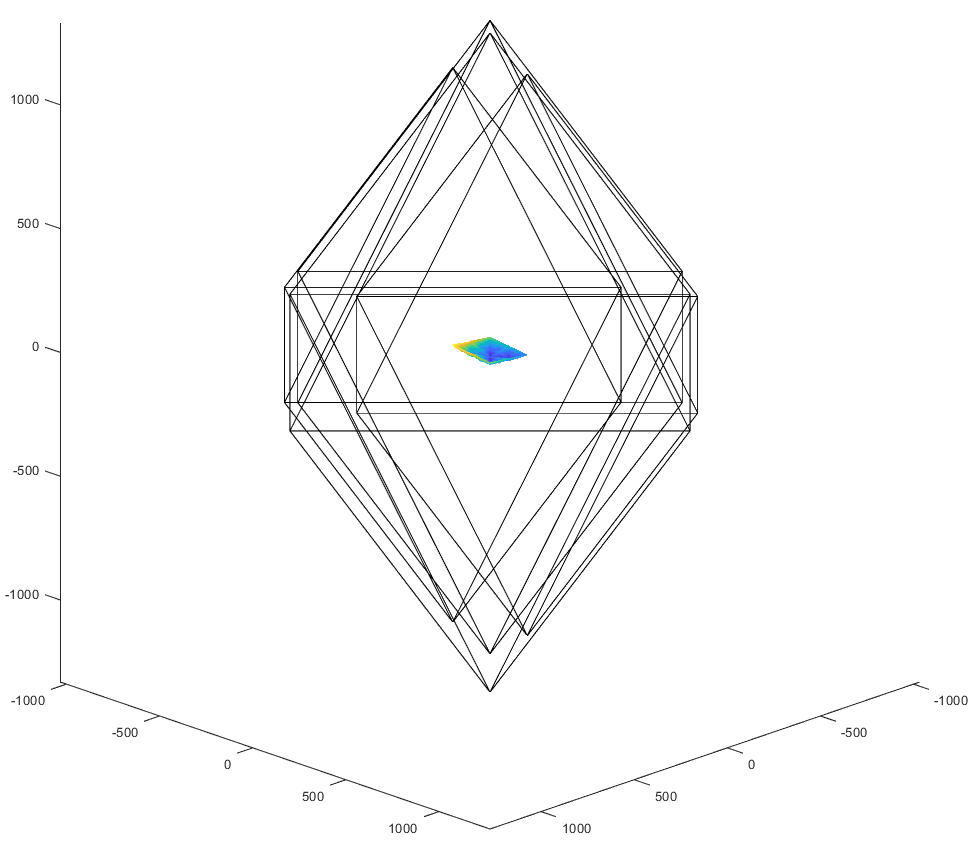}} &
    \subfloat[\((II):\) Visualization of \(\mathcal{C}_3\)]{\includegraphics[width = 0.5\textwidth]{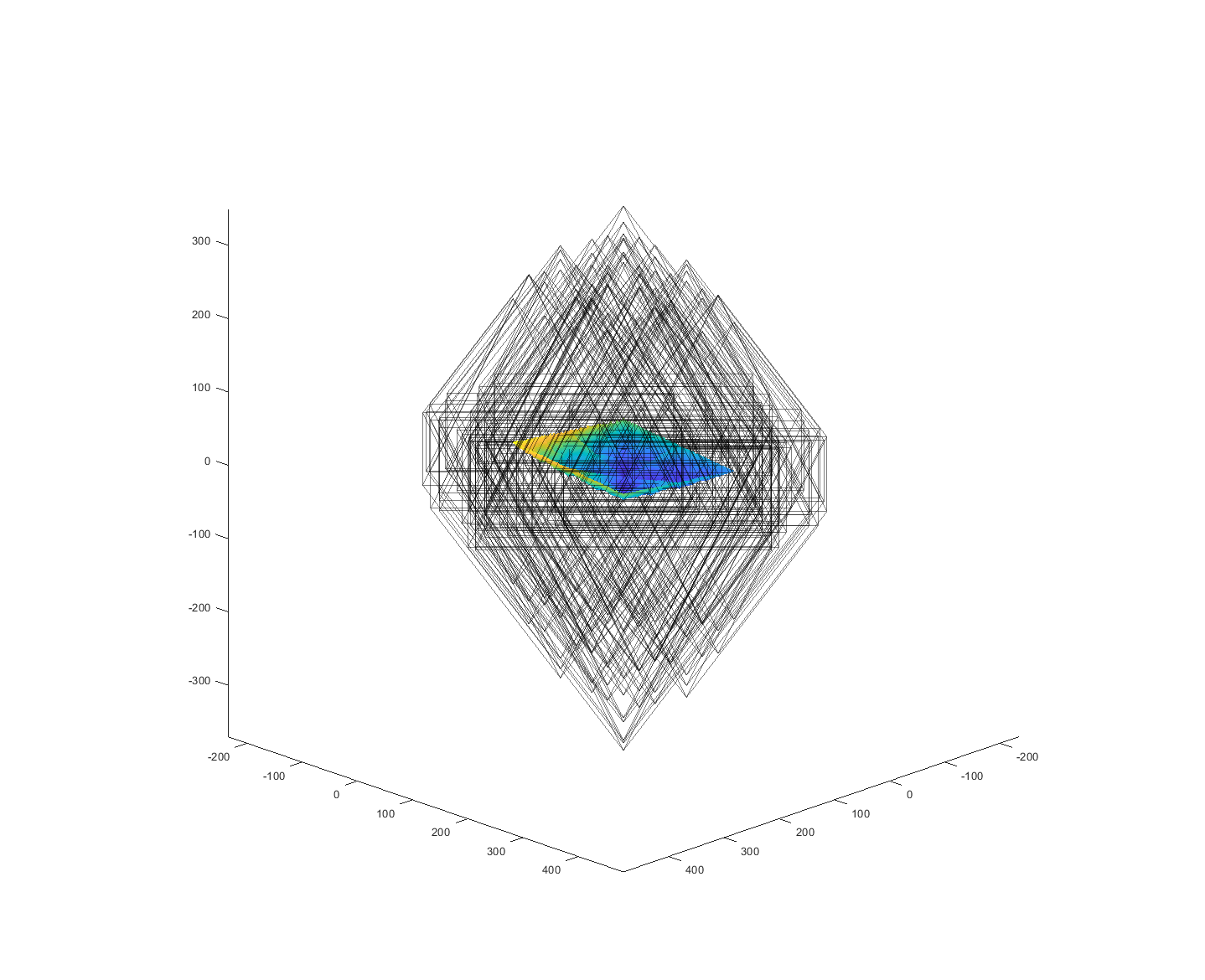}}
  \end{tabular}
\end{figure}

\begin{figure}[H]
  \centering
  \begin{tabular}{cc}
    \subfloat[\((III):\) Visualization of \(\mathcal{C}_5\)]{\includegraphics[width = 0.5\textwidth]{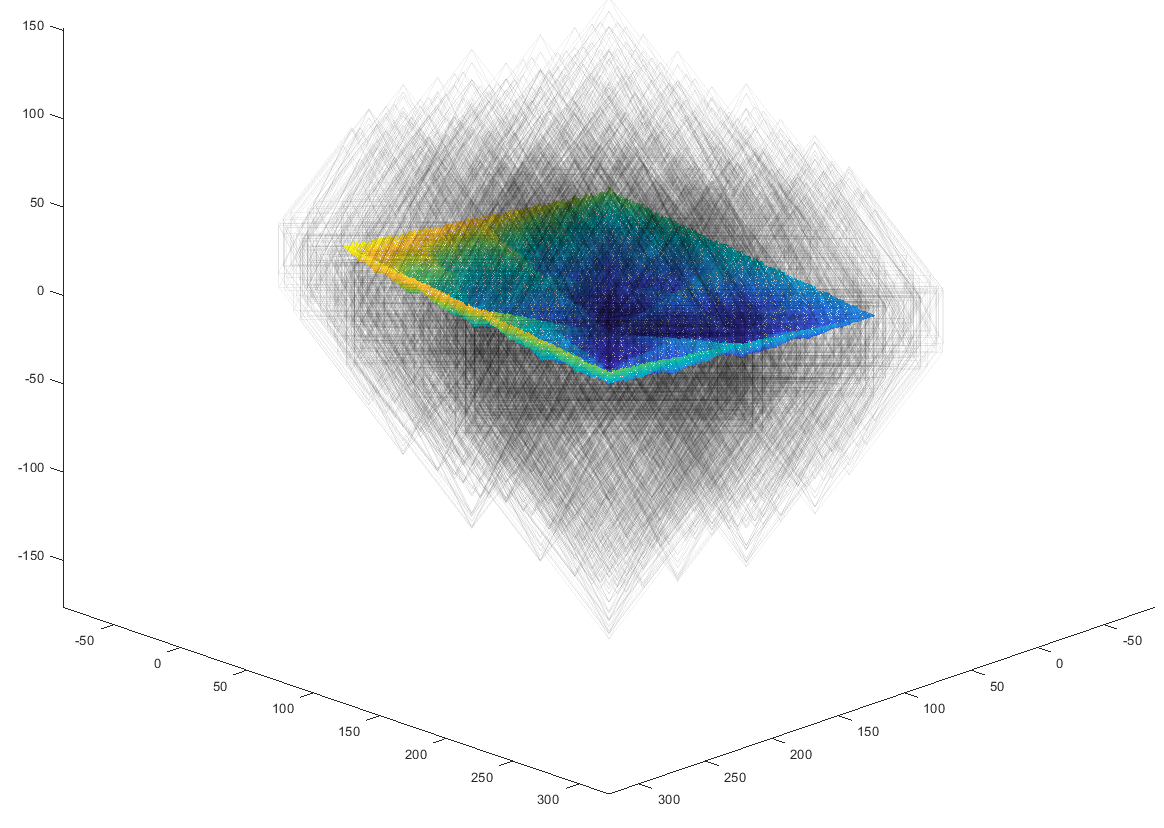}} &
    \subfloat[\((IV):\) Visualization of \(\mathcal{C}_7\)]{\includegraphics[width = 0.5\textwidth]{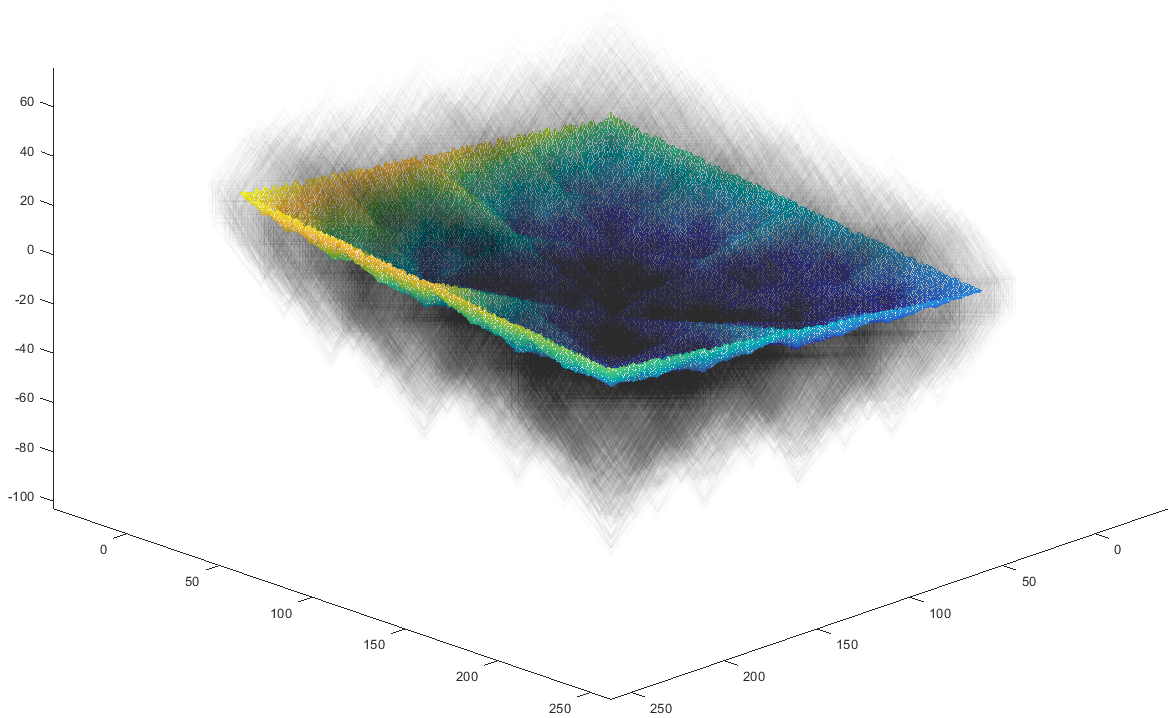}}
  \end{tabular}
\end{figure}

\begin{figure}[H]
  \centering
  \begin{tabular}{cc}
    \subfloat[\((V):\) Visualization of \(\mathcal{C}_9\)]{\includegraphics[width = 0.5\textwidth]{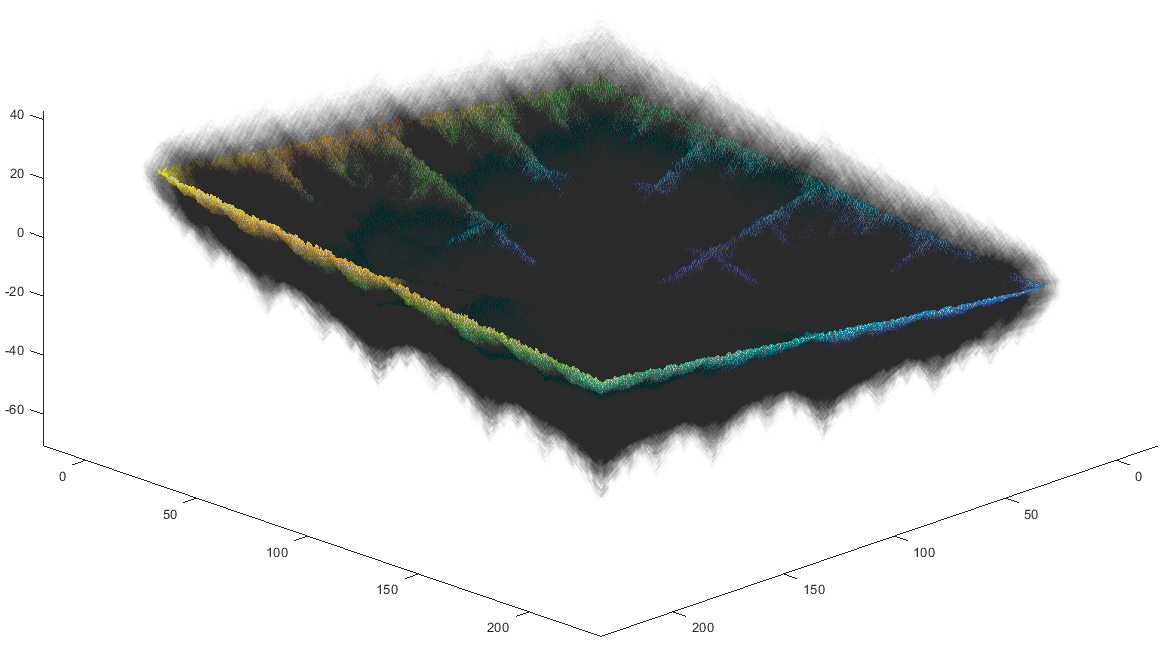}} &
  \end{tabular}
\end{figure}

\textbf{Example 2}
\\

Let
\[x_0=0,\;x_1=100,\;x_2=200, \;x_3=300\]
\[y_0=0,\;y_1=100,\;y_2=200,\; y_3=300\]
with the corresponding \(z_{k,l}\) values 
  \begin{center}
        \begin{tabular}{ c|c|c|c|c }
            \backslashbox{\(l\)}{\(k\)}& 0 & 1& 2& 3 \\ 
            \hline 0 & 0 & 15 & 30 & 45\\
            \hline 1& -10 & 20 & -30 & 35\\
            \hline 2 & -20 & 30 & 10 & 25\\
            \hline 3 & -30 & -15 & 0 & 15
           \end{tabular}
    \end{center}
and \(g_{k,l}\) given by

\begin{center}
    \begin{tabular}{ c|c|c|c }
        \backslashbox{\(l\)}{\(k\)}& 1& 2& 3 \\ 
        \hline 1 & 0.3 & 0.2& 0.5 \\
        \hline 2& 0.4 & 0.7 & 0.6 \\
        \hline 3& 0.3 & 0.6 & 0.4
     \end{tabular}.
\end{center}

The interpolation surface has the following graphical representation:

\captionsetup[subfigure]{labelformat=empty}
\begin{figure}[H]
  \centering
  \begin{tabular}{c}
    \subfloat[\((O):\) Visualization of the graph of the interpolated surface]{\includegraphics[width = 0.5\textwidth]{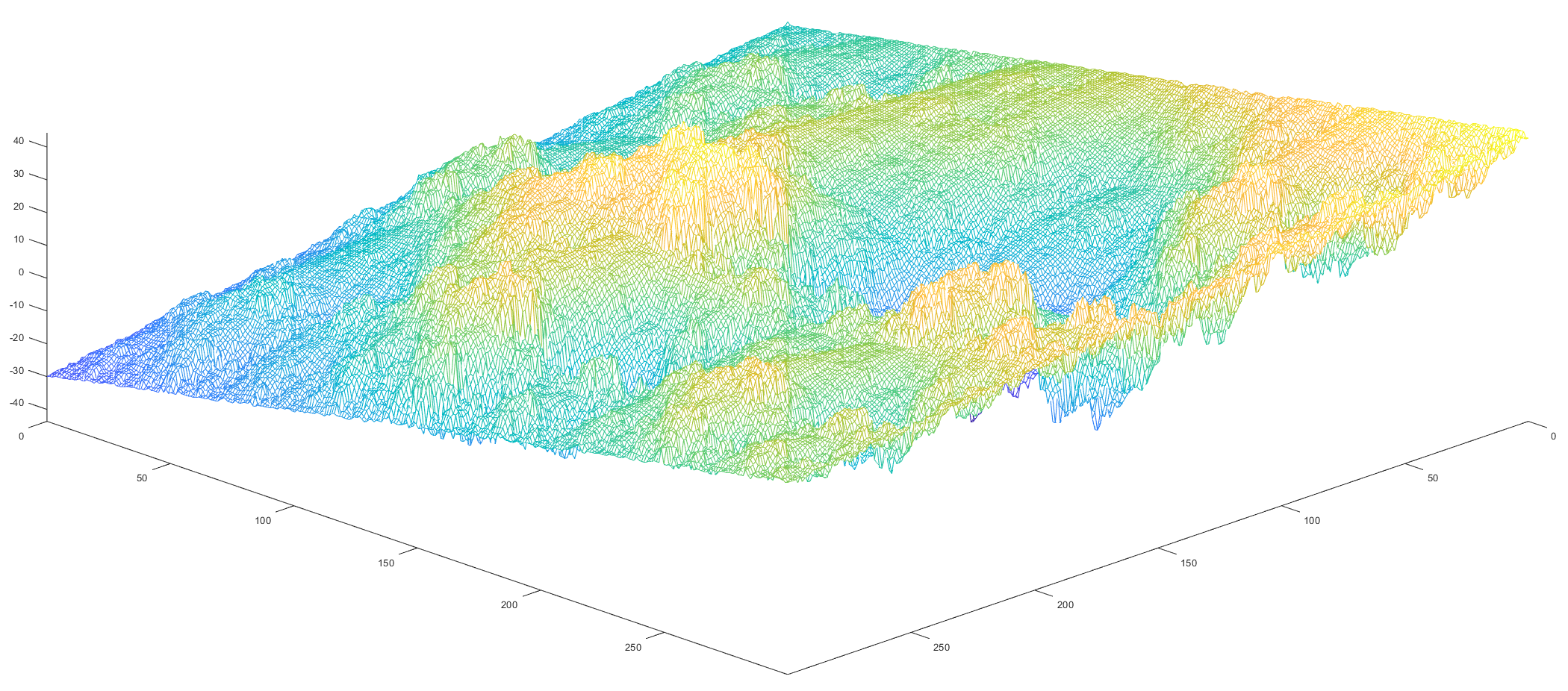}} 
  \end{tabular}
\end{figure}

The figures \((I),\) \((II),\) \((III),\) \((IV),\) \((V),\) \((VI)\) contain graphical representations of the covers \(\mathcal{C}_1\), \(\mathcal{C}_2\), \(\mathcal{C}_3\), \(\mathcal{C}_4\) and \(\mathcal{C}_5\).
\captionsetup[subfigure]{labelformat=empty}
\begin{figure}[H]
  \centering
  \begin{tabular}{cc}
    \subfloat[\((I):\) Visualization of \(\mathcal{C}_1\)]{\includegraphics[width = 0.5\textwidth]{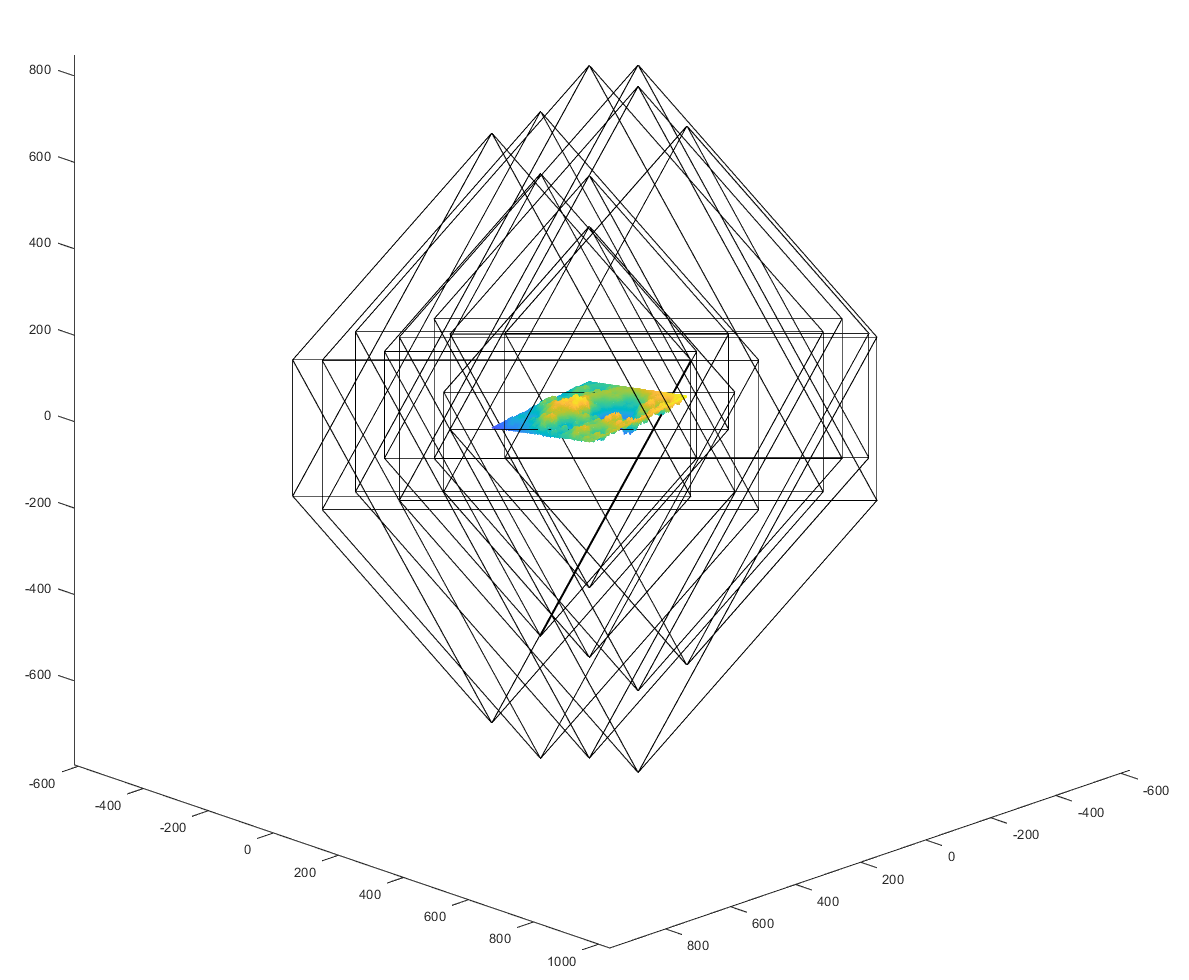}} &
    \subfloat[\((II):\) Visualization of \(\mathcal{C}_2\)]{\includegraphics[width = 0.5\textwidth]{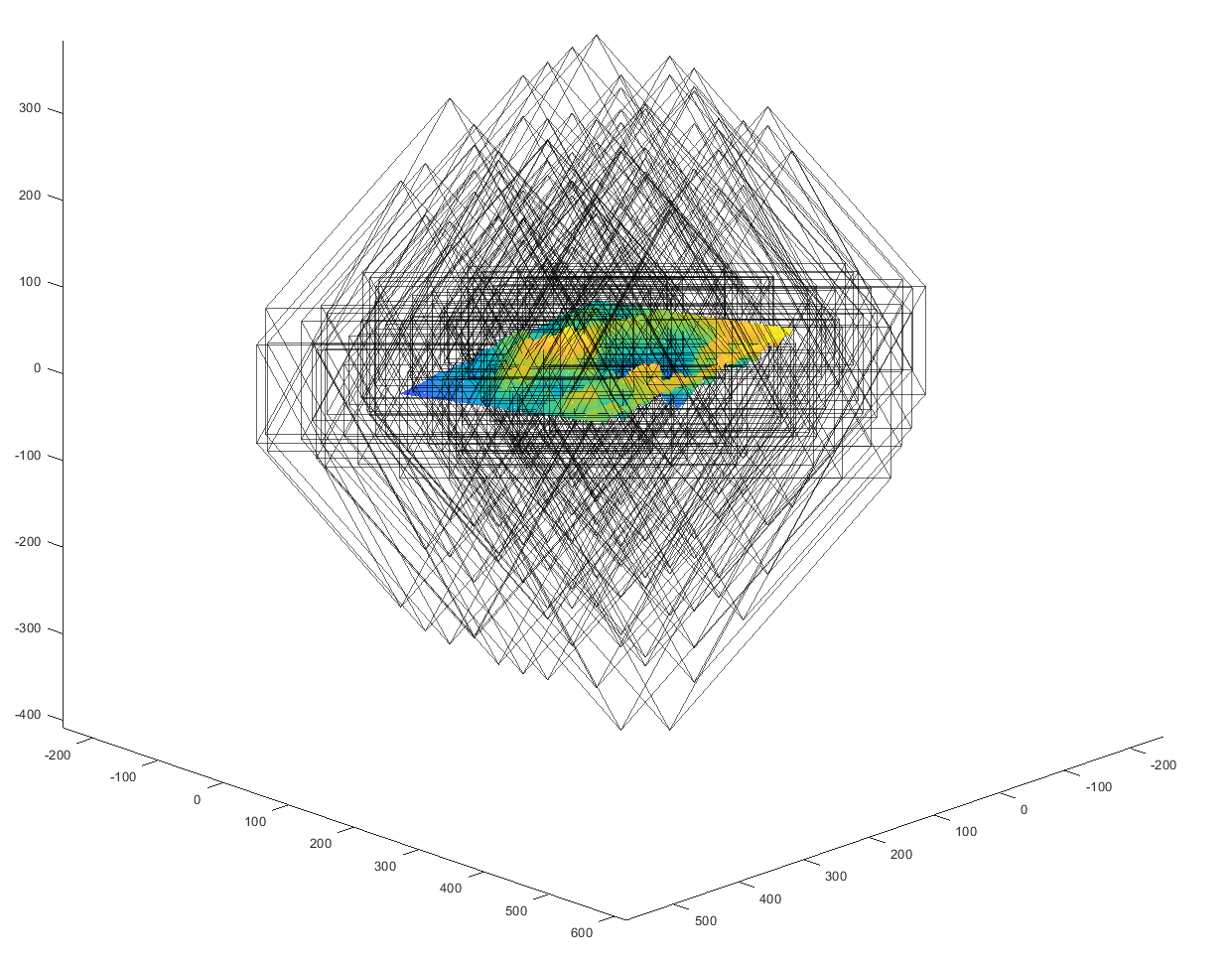}}
  \end{tabular}
\end{figure}

\begin{figure}[H]
  \centering
  \begin{tabular}{cc}
    \subfloat[\((III):\) Visualization of \(\mathcal{C}_3\)]{\includegraphics[width = 0.5\textwidth]{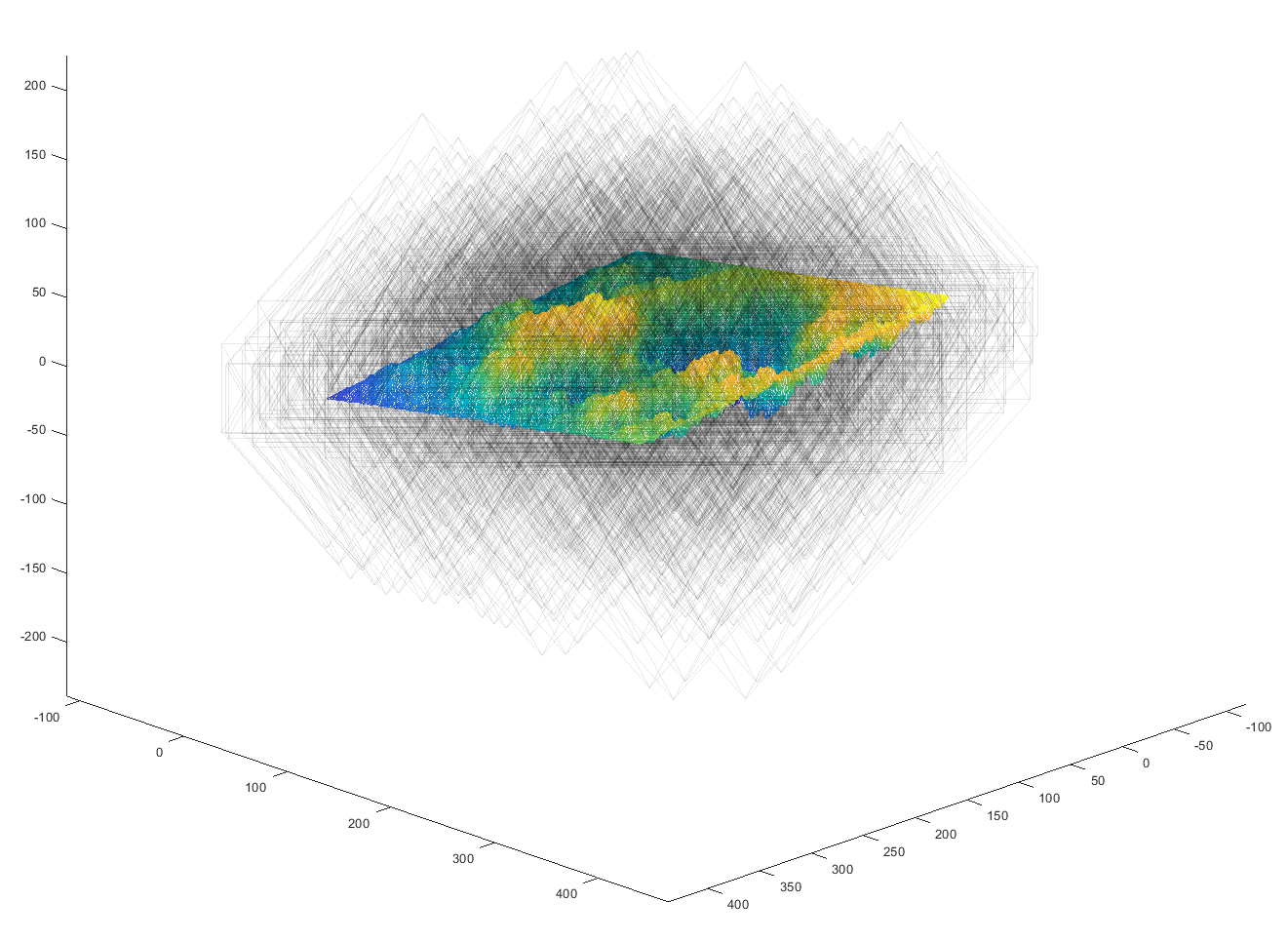}} &
    \subfloat[\((IV):\) Visualization of \(\mathcal{C}_4\)]{\includegraphics[width = 0.5\textwidth]{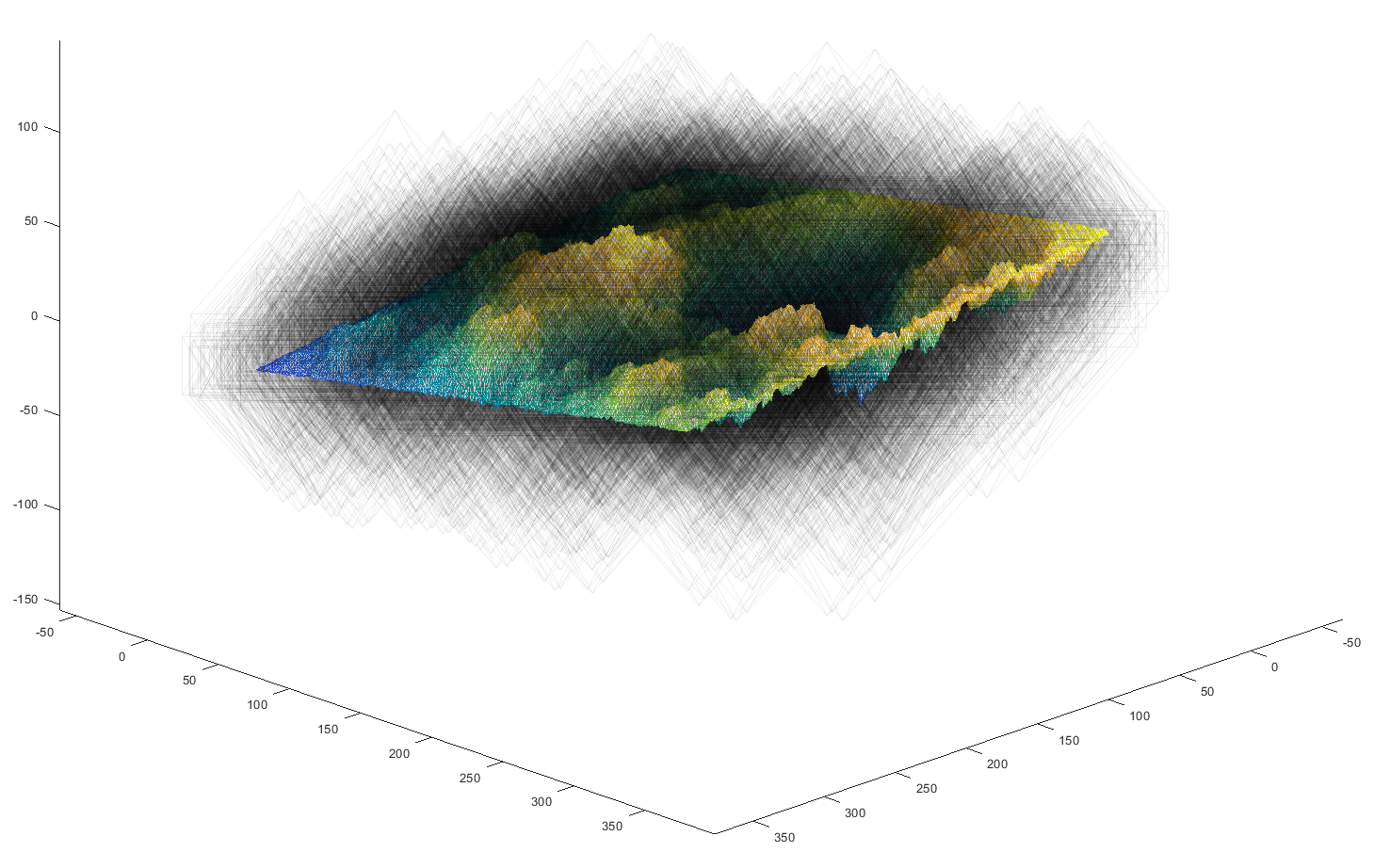}}
  \end{tabular}
\end{figure}

\begin{figure}[H]
  \centering
  \begin{tabular}{cc}
    \subfloat[\((V):\) Visualization of \(\mathcal{C}_5\)]{\includegraphics[width = 0.5\textwidth]{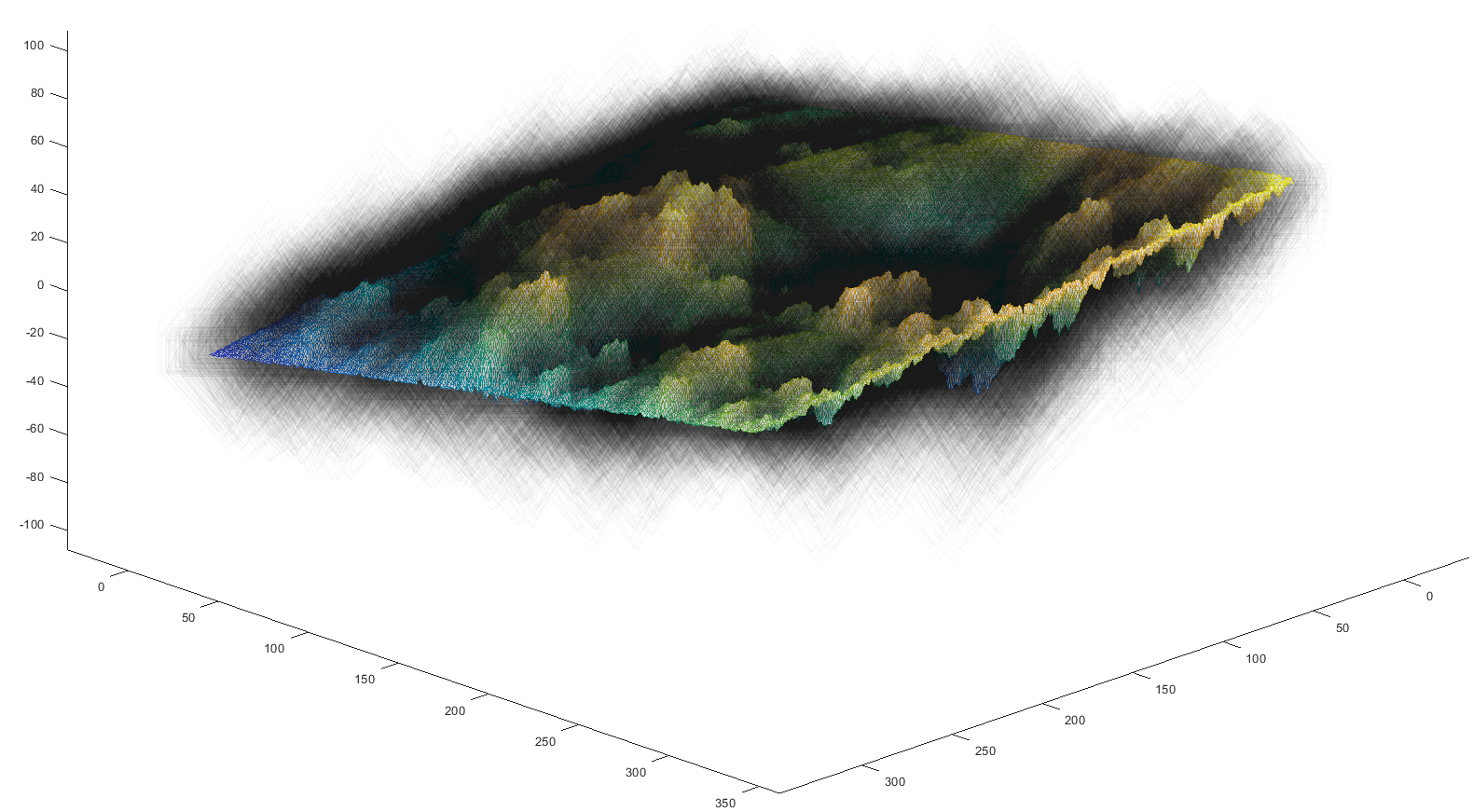}} &
  \end{tabular}
\end{figure}

\pagebreak
\section*{Appendix}
We provide the following algorithms, which were used to determine the covers presented in the previous examples.

\begin{algorithm}[H]
    \caption{Finds the coefficients of the functions}
    \KwIn{\(x=[x_0,\dots ,x_n]\), \(y=[y_0,\dots, y_m]\), \(z =\left[
        \begin{aligned}
        &z_{0,0},\dots,z_{0,m}\\
        &\vdots\qquad\qquad\vdots\\
        &z_{n,0},\dots, z_{n,m}
    \end{aligned}\right],\)
    \(g =\left[
        \begin{aligned}
        &g_{1,1},\dots,g_{1,m}\\
        &\vdots\qquad\qquad\vdots\\
        &g_{n,1},\dots, g_{n,m}
    \end{aligned}\right].\)   } 
    
    \KwOut{ \(\theta\), 
    \(F=\left[
        \begin{aligned}
        &F_{1,1},\dots, F_{1,m}\\
        &\vdots\qquad\qquad\vdots\\
        &F_{n,1},\dots, F_{n,m}
    \end{aligned}\right]\).}
    \SetKwFunction{Fcoef}{coeff}
    \SetKwProg{Fn}{function}{:}{}
    \Fn{\Fcoef{\(x\), \(y\), \(z\), \(g\)} }
    {\For{\(i\gets 1\) \KwTo \(n\)}{
        \(F(i,j)(1)\gets(x(i)-x(i-1))/(x(n)-x(0));\)

        \(F(i,j)(2)\gets(x(i)*x(n)-x(i-1)*x(0))/(x(n)-x(0));\)

        \For{\(j\gets 1\) \KwTo \(m\)}{
    
        \(p=z(i,j)-g(i,j)*z(n,m);\)

        \(q=z(i-1,j)-g(i,j)*z(0,m);\)

        \(r=z(i,j-1)-g(i,j)*z(n,0);\)

        \(t=z(i-1,j-1)-g(i,j)*z(0,0);\)

        \(F(i,j)(3)\gets(y(j)-y(j-1))/(y(n)-y(0));\)

        \(F(i,j)(4)\gets(y(j)*y(n)-y(j-1)*y(0))/(y(n)-y(0));\)
        
        \(F(i,j)(5) \gets (y(0)*(q-p)-y(m)*(t-r))/((x(n)-x(0))*(y(m)-y(0)));\)

        \(F(i,j)(6)\gets (x(0)*(r-p)-x(n)*(t-q))/((x(n)-x(0))*(y(m)-y(0)));\)

        \(F(i,j)(7)\gets g(i,j);\)

        \(F(i,j)(8)\gets (p-q-r+t)/((x(n)-x(0))*(y(m)-y(0)));\)

        \(F(i,j)(9) \gets (y(0)*(x(0)*p-x(n)*q)-y(m)*(x(0)*r\)
        
        \hspace{180pt}\(-x(n)*t))/((x(n)-x(0))*(y(m)-y(0)));\)
        }
    }
    \KwRet{F};
    }

\end{algorithm}

\begin{algorithm}[H]
    \caption{Calculates \(\theta\), the constants and fixed point associated to each function}

    \KwIn {\(\delta\),  
    \(F=\left[
        \begin{aligned}
        &F_{1,1},\dots, F_{1,m}\\
        &\vdots\qquad\qquad\vdots\\
        &F_{n,1},\dots, F_{n,m}
    \end{aligned}\right]\)}
    \tcc{Where \(F_{1,1},\dots,F_{k,l}\) are arrays containing the coefficients corresponding to each function and \(\delta\) is \(\max\{|x_0|,|x_n|,|y_0|,|y_n|\}\).}
    \KwOut{\(\theta\), \(\gamma=\left[
            \begin{aligned}
            &\gamma_{1,1},\dots, \gamma_{1,m}\\
            &\vdots\qquad\qquad\vdots\\
            &\gamma_{n,1},\dots, \gamma_{n,m}
        \end{aligned}\right]\),
        \(c =\left[
        \begin{aligned}
        &c_{1,1},\dots,c_{1,m}\\
        &\vdots\qquad\qquad\vdots\\
        &c_{n,1},\dots, c_{n,m}.
    \end{aligned}\right]\)}
    \tcc{Where \(c_{1,1},\dots,c_{k,l}\) are the constants corresponding to each function \(F_{1,1},\dots,F_{k,l}\), \(\gamma\) contains their fixed points and \(\theta\) is the constant considered in the metric.}
    \SetKwProg{Fn}{function}{:}{}

    \SetKwFunction{Ftht}{theta}
        \Fn{\Ftht{\(F\), \(\delta\)}}{
            \For{\(i\gets 1\) \KwTo \(n\) and \(j\gets 1\) \KwTo \(m\)}{  
            [\(a,b,c,d,e,f,g,\alpha,\beta]\gets\) F(i, j);
            
            \(max_a \gets \max(max_a,a)\);

            \(max_c\gets \max(max_c,c)\);

            \(max_{1}\gets \max(max_1,|e|+\delta*|\alpha|)\);

            \(max_{2}\gets \max(max_2,|f|+\delta*|\alpha|)\);
            }
            \(\theta_1=\theta_2=1\)

            \If{\(max_1\neq 0\)}{\(\theta_1 \gets(1-max_a)/(2*max_1) \)}
            \If{\(max_2\neq 0\)}{\(\theta_2\gets(1-max_c)/(2*max_2) \)}
            
            \(\theta = \min(\theta_1,\theta_2);\)

            \KwRet{\(\theta\);}
            }

    \SetKwFunction{Ffp}{fixPT}
    \Fn{\Ffp{\(F\)}}{
        \For{\(i\gets 1\) \KwTo \(n\) and \(j\gets 1\) \KwTo \(m\)}{
        \([a,b,c,d,e,f,g,\alpha,\beta]\gets F(i,j);\)

        \(\gamma(i,j)\gets[b/(1-a), d/(1-c),(1/(1-g(i,j)))*(e*b/(1-a) \)
        
        \hspace{80pt}\(+f*d/(1-c) +\alpha*b*d/((1-a)*(1-c)) +\beta)];\)}
       
        \KwRet{\(\gamma\);}
        }
        \SetKwFunction{Flip}{const}
        \Fn{\Flip{\(F\), \(\theta\)}}{
            \For{\(i\gets 1\) \KwTo \(n\) and \(j\gets 1\) \KwTo \(m\)}{
            \([a,b,c,d,e,\alpha,\beta]\gets F(i,j);\)
    
            \(c(i,j)\gets \max(a+\theta*(|e|+\delta|\alpha|),c+\theta*(|f|+\delta|\alpha|),g(i,j)));\)}
           
            \KwRet{\(\theta\);}
            }     
\end{algorithm}
\begin{algorithm}[H]
    \caption{Calculates coefficients of the composed functions and their corresponding constants}
    \KwIn{\(F_0=\left[
        \begin{aligned}
        &F_{1,1},\dots, F_{1,m}\\
        &\vdots\qquad\qquad\vdots\\
        &F_{n,1},\dots, F_{n,m}
    \end{aligned}\right]\), \(c_0 =\left[
        \begin{aligned}
        &c_{1,1},\dots,c_{1,m}\\
        &\vdots\qquad\qquad\vdots\\
        &c_{n,1},\dots, c_{n,m}.
    \end{aligned}\right]\), \(order\)}
    
    \KwOut{
        \(F_n,c_n\)}
    \SetKwProg{Fn}{function}{:}{}

    \SetKwFunction{Fcomp}{comp2}
        \Fn{\Fcomp{\([a_1, b_1, c_1, d_1, e_1, f_1, g_1, \alpha_1, \beta_1]\), \([a_2, b_2, c_2, d_2, e_2, f_2, g_2, \alpha_2, \beta_2]\)}}{
            \(a\gets a_1*a_2\);

            \(b\gets a_1*b_2+b_1;\)
            
            \(c\gets c_1* c_2;\)

            \(d\gets c_1*d_2+d_1;\)

            \(e\gets e_1*a_2+g_1*e_2+\alpha_1*a_2*d_2;\)

            \(f\gets f_1*c_2+g_1*f_2+\alpha_1*b_2*c_2;\)

            \(g\gets g_1*g_2;\)

            \(\alpha\gets \alpha_1*a_2*c_2+g_1*\alpha_2 ;\)

            \(\beta\gets e_1*b_2+f_1*d_2+\alpha_1*b_2*d_2+g_1*\beta_2+\beta_1;\)

            \KwRet{\([a,b,c,d,e,f,g,\alpha,\beta]\);}
        }

        \SetKwFunction{Fcomps}{sisComp}
        \Fn{\Fcomps{\(F_o\), \(F_n\), \(c_o\), \(c_n\)}}{
            
            \(index_1\gets 0;\)

            \For{\(i\gets 1\) \KwTo \(n\) and \(j\gets 1\) \KwTo \(m\)}{
                \(index_1\gets index_1+1;\)

                \(index_2\gets0;\)

            \For{\(k\gets 1\) \KwTo \(n'\) and \(l\gets 1\) \KwTo \(m'\)}{
                 \(index_2\gets index_2+1;\)

                 \(F(index_1,index_2)=comp2(F_o(i,j),F_n(k,l));\)

                 \(c(index_1,index_2)=c_o(i,j)*c_n(k,l);\)
            }
            }
            \KwRet{\(F,c\);}
        }
        \SetKwFunction{Ford}{sisCompOrd}
        \Fn{\Ford{\(F_o\),\(c_o\),order}}{
            
        \If{order\(\geq 2\)}{\([F_n,c_n]\gets\)sisComp(\(F_o,F_o,c_o,c_o\));
        
        \For{\(i\gets 3\) \KwTo order}{
            \([F_n,c_n]\gets\)sisComp(\(F_o,F_n,c_o,c_n\));
        }
        }
        \Else{\KwRet{\(F_o,c_o\)};}
            \KwRet{\(F_n,c_n\);}
        }
        
    \end{algorithm}
\begin{algorithm}[H]
    \caption{Finds the radii and vertices of the covering}
    \KwIn{\(k,l \geq 2,\;c =\left[
        \begin{aligned}
        &c_{1,1},\dots,c_{1,l}\\
        &\vdots\qquad\qquad\vdots\\
        &c_{k,1},\dots, c_{k,l}
    \end{aligned}\right],\;
    \gamma=\left[
        \begin{aligned}
        &\gamma_{1,1},\dots, \gamma_{1,l}\\
        &\vdots\qquad\qquad\vdots\\
        &\gamma_{k,1},\dots, \gamma_{k,l}
    \end{aligned}\right],\; M \).   } 
    \tcc{Where \(c_{1,1},\dots,c_{k,l}\) are the constants corresponding to each function \(F_{1,1},\dots,F_{k,l}\), \(\gamma\) contains the fixed points and \(M\) is 
    \(\max_{\ovset{i,p\in\{1,\dots,k\}}{l,q\in\{1,\dots,l\}}}\rho(\gamma_{i,j},\gamma_{k,l})\).}
    
    \KwOut{\(r = \left[
        \begin{aligned}
        &r_{1,1},\dots,r_{1,l}\\
        &\vdots\qquad\qquad\vdots\\
        &r_{k,1},\dots, r_{k,l}
    \end{aligned}\right]\), \(O = \left[
        \begin{aligned}
        &o_{1,1},\dots,o_{1,l}\\
        &\vdots\qquad\qquad\vdots\\
        &o_{k,1},\dots,o_{k,l}
    \end{aligned}\right]\)}
    \tcc{\(r\) contains the radii and \(O\) is a matrix in which each element \(o_{i,j}\) contains the eight vertices that define the  octahedron corresponding to each function in the system.
    }

    \SetKwFunction{Fradi}{radii}
    \SetKwProg{Fn}{function}{:}{}
    \Fn{\Fradi{\(k\), \(l\), \(c\), \(M\)} }
    {

    \(max_c=[0,0];\)
    
    \(ind_1=[0,0];\)
    
    \(ind_2=[0,0];\)

    \For{\(i\gets 1\) \KwTo \(k\) and \(j\gets 1\) \KwTo \(l\)}{
        \If{$max_c(1)\leq c(i,j)$}{
            \(max_c(2)\gets max_c(1)\)\;
            \(ind_2\gets ind_1\)\;
            \(max_c(1)\gets c(i,j)\)\;
            \(ind_1\gets [i,j]\)\;
        }
    }
    \(r(ind_1)\gets c(ind_1)* M*(1+c(ind_2))/(1-c(ind_2)*c(ind_1)) \)\;
    \For{\(i\gets 1\) \KwTo \(k\) and \(j\gets 1\) \KwTo \(l\), \([i,j]\neq ind_1\)}{
        \(r(i,j)\gets c(i,j)* M*(1+c(ind_1))/(1-c(ind_2)*c(ind_1)) \)\;
    }
    \KwRet{r};
    }

    \SetKwFunction{Fvert}{vertices}
    \SetKwProg{Fn}{function}{:}{}
    \Fn{\Fvert{\(k\), \(l\), \(\gamma\), \(r\), \(\theta\)} }
    {
        \For{\(i\gets 1\) \KwTo \(k\) and \(j\gets 1\) \KwTo \(l\)}{
            \(O(i,j)\gets\)\([\gamma(i,j)+[r(i,j),0,0]\), \(\gamma(i,j)-[r(i,j),0,0]\),
                        
            \Indp\Indp\Indp\(\gamma(i,j)+[0,r(i,j),0]\), \(\gamma(i,j)-[0,r(i,j),0]\),
                    
            \(\gamma(i,j)+[0,0,r(i,j)/\theta]\), \(\gamma(i,j)-[0,0,r(i,j)/\theta]];\)
            
            \tcc{Array addition is done element by element}
        }
    \KwRet{O};

    }
    \end{algorithm}

\makeatletter
\def\verbatim@font{\normalfont}
\makeatother
\begin{verbatim}
Bogdan-Cristian Anghelina
Faculty of Mathematics and Computer Science
Transilvania University of Brașov
Iuliu Maniu Street, nr. 50, 500091, Brașov, Romania
E-mail: bogdan.anghelina@unitbv.ro

Radu Miculescu
Faculty of Mathematics and Computer Science
Transilvania University of Brașov
Iuliu Maniu Street, nr. 50, 500091, Brașov, Romania
E-mail: radu.miculescu@unitbv.ro
\end{verbatim}
%





\begin{thebibliography}{00}
    \bibitem{locHBfrct} B. Anghelina, R. Miculescu, On the localization of Hutchinson-Barnsley fractals, Chaos Solitons Fractals, \textbf{173} (2023), 113-674.


    \bibitem{bouboulis1}P. Bouboulis, L. Dalla, Closed fractal interpolation surfaces, J. Math.
    Anal. Appl., \textbf{327} (2007), 116--126.

    \bibitem{bouboulis2}P. Bouboulis, L. Dalla, Fractal interpolation surfaces derived from fractal
    interpolation functions, J. Math. Anal. Appl., \textbf{336} (2007), 919--936.

    \bibitem{bouboulis3} P. Bouboulis, L. Dalla, V. Drakopoulos, Image compression using recurrent
    bivariate fractal interpolation surfaces, Internat. J. Bifur. Chaos, \textbf{%
        16} (2006), 2063-2071.
    \bibitem{cambell} B. Cambell, M. Shepard, Shadows on a planetary surface and implications for
    photometric roughness, ICARUS, \textbf{134} (1998), 279--291.

    \bibitem{chand1} A.K.B. Chand, N. Vijender, A new class of fractal interpolation surfaces
    based on functional values, Fractals, \textbf{24} (2016), 1650007, 17 pp.

    \bibitem{chand2} A.K.B. Chand, K. Tyada, Partially blended constrained rational cubic
    trigonometric fractal interpolation surfaces, Fractals, \textbf{24} (2016),
    1650027, 21 pp.

   

    \bibitem{dallas} L. Dalla, Bivariate fractal interpolation functions on grids, Fractals, \textbf{10} (2002), 53-58.

    \bibitem{drakopoulos} V. Drakopoulos, P. Manousopoulos, Bivariate fractal interpolation surfaces:
    theory and applications, Internat. J. Bifur. Chaos Appl. Sci. Engrg.,
    \textbf{22} (2012), 1250220, 8 pp.

    \bibitem{feng} Z. Feng, Variation and Minkowski dimension of fractal interpolation surface,
    J. Math. Anal. Appl., \textbf{345} (2008), 322--334.

    \bibitem{feng2} Z. Feng, Y. Feng, Z. Yuan, Fractal interpolation surfaces with function
    vertical scaling factors, Appl. Math. Lett., \textbf{25} (2012), 1896--1900.

    \bibitem{geronimo}J. Geronimo, D. Hardin, Fractal interpolation surfaces and a related 2D multiresolution analysis, J. Math. Anal. Appl., \textbf{176} (1993), 346-355.

    \bibitem{hardin} D. Hardin, P. Massopust, Fractal interpolation functions from $\mathbb{R}%
        ^{n}\rightarrow \mathbb{R}^{m}$ and their projections, Z. Anal. Anw.,
    \textbf{12} (1993), 535--548.

    \bibitem{liang} Z. Liang, H. Ruan, Recurrent fractal interpolation surfaces on triangular
    domains, Fractals, \textbf{27} (2019), 1950085, 12 pp.

    \bibitem{liang2} Z. Liang, H. Ruan, Construction and box dimension of recurrent fractal
    interpolation surfaces, J. Fractal Geom., \textbf{8} (2021), 261--288.

    \bibitem{malysz} R. Ma\l ysz, The Minkowski dimension of the bivariate fractal interpolation
    surfaces, Chaos Solitons Fractals, \textbf{27} (2006), 1147--1156.

    \bibitem{massopus1} P. Massopust, Fractal Surfaces, J. Math. Anal. Appl., \textbf{151} (1990), 275-290.


    \bibitem{metzler} W. Metzler, C. Yun, Construction of fractal interpolation surfaces on
    rectangular grids, Internat. J. Bifur. Chaos, \textbf{20} (2010), 4079--4086.

    \bibitem{ri1} S. Ri, A new construction of the fractal interpolation surface, Fractals,
    \textbf{23} (2015), 1550043, 12 pp.

    \bibitem{ri2} S. Ri, New types of fractal interpolation surfaces, Chaos Solitons Fractals,
    \textbf{123} (2019), 52--58.

    \bibitem{ruan} H. Ruan, Q. Xu, Fractal interpolation surfaces on rectangular grids, Bull.
    Aust. Math. Soc., \textbf{91} (2015), 435--446.

    
    \bibitem{xie1} H. Xie, H. Sun, The study of bivariate fractal interpolation functions and
    creation of fractal interpolated surfaces, Fractals, \textbf{5} (1997),
    625-634.

    \bibitem{xie2} H. Xie, H. Sun, Y. Zu, Z. Feng, Study on generation of rock fracture
    surfaces by using fractal interpolation, Internat. J. Solids Struct.,
    \textbf{38} (2001), 5765--5787.


    \bibitem{yokoya} N. Yokoya, K. Yamamoto, N. Funakubo, Fractal-based analysis and
    interpolation of 3d natural surface shapes and their application to terrain
    modeling, Comput. Vision Graphics Image Process, \textbf{46} (1989) 284--302.

    \bibitem{yun} C. Yun, H. O, H. Choi, \ Construction of fractal surfaces by recurrent
    fractal interpolation curves, Chaos Solitons Fractals, \textbf{66} (2014),
    136--143.

    \bibitem{zhao} N. Zhao, Construction and application of fractal interpolation surfaces, The Visual Computer, \textbf{12} (1996), 132-146.

\end{thebibliography}
\end{document}